\documentclass[11pt]{amsart}

\usepackage{amsthm}
\usepackage{amsmath}
\usepackage{amssymb}
\usepackage{color}
\usepackage{enumerate}

\newtheorem{thm}{Theorem}[section]
\newtheorem{theorem}[thm]{Theorem}

\newtheorem{corollary}[thm]{Corollary}

\newtheorem{observation}[thm]{Observation}

\newtheorem{lemma}[thm]{Lemma}
\newtheorem{proposition}[thm]{Proposition}

\newtheorem{definition}[thm]{Definition}
\theoremstyle{remark}
\newtheorem{remark}[thm]{Remark}

\newcommand{\RR}{\mathbb R}

\newcommand{\HH}{\mathcal H}

\newcommand{\ip}[2]{\left\langle#1,#2\right\rangle}

\begin{document}

\title[The distribution of frame vectors and coefficients]
{The distributions of Hilbert space frame vectors and frame coefficients}
\author[Brewster, Casazza, Pinkham and Woodland]{Kevin Brewster, Peter G. Casazza, Eric Pinkham and 
 Lindsey Woodland}
\address{Department of Mathematics, University
of Missouri, Columbia, MO 65211-4100}

\thanks{The authors were supported by
 NSF DMS 1307685; NSF ATD 1321779; and AFOSR  DGE51:  FA9550-11-1-0245}

\email{kjb886@missouri.edu, casazzap@missouri.edu,\newline eap9qc@missouri.edu, lmwvh4@missouri.edu}



\begin{abstract}
The most fundamental notion for Hilbert space frames is the sequence of frame coefficients for a vector $x$ in the space. Yet, we know little about the distribution of these coefficient sequences.  In this paper, we make the first detailed study of the distribution of the frame coefficients 
for vectors in a Hilbert space,
as well as the
related notion of the square sums of the distances of vectors in the space from the frame vectors.  
We will give some surprisingly exact calculations for special frames such as unit norm tight frames and equiangular frames.  This is a study which should have
been done 20 years ago.

\end{abstract}

\maketitle

\section{Introduction}

The most fundamental notion for a Hilbert space frame 
$\Phi=\{\varphi_i\}_{i=1}^N$ for $\HH_M$ is the sequence
of frame coefficients for a vector x, namely $\{\ip{x}{\varphi_i}\}_{i=1}^N$. Yet, we know little about the distribution of these coefficients, even for very specific frames.  In this paper, we will make the first detailed study of the distribution of the frame coefficients for general frames and then strengthen the results for several special classes of frames including unit norm tight frames
and equiangular frames. For this we use the concept of \emph{majorization}.
We will also study the distributions of products of frame coefficients for different vectors 
$x$ and $y$.
We will then make a detailed study of the square sums of the distances from a vector x to the frame vectors and discover that in quite general cases, these sums are nearly equal for all vectors.


\section{Preliminaries}\label{sec2}

While the following work is largely self contained, an interested reader can find a more detailed introduction to frame theory in \cite{CK,OLE,FFU}. Given $M\in\mathbb{N}$, $\HH_{M}$ represents a $($real or complex$)$ Hilbert space of $($finite$)$ dimension $M$.

\begin{definition}
A family of vectors $\{\varphi_i\}_{i=1}^N$ in an $M$-dimensional Hilbert space $\mathcal{H}_M$ is a {\bf frame} if there are constants $0 < A \leq B < \infty$ so that for all $x \in \mathcal{H}_M$,
\begin{equation*} 
A \| x \|^2 \leq \sum_{i=1}^N |\langle x,\varphi_i\rangle|^2 \leq B \|x\|^2,
\end{equation*}
where $A$ and $B$ are {\bf lower and upper frame bounds}, respectively.  The largest $A$ and smallest $B$ satisfying these inequalities are called the {\bf optimal frame bounds}.

\begin{enumerate}[(i)]

\item If $A = B$ is possible, then $\{\varphi_i\}_{i=1}^N$ is a {\bf A-tight frame}. If $A = B = 1$ is possible, then $\{\varphi_i\}_{i=1}^N$ is a {\bf Parseval frame}.

\item If $\|\varphi_i\| = 1$ for all $i \in [N]$ then $\{\varphi_i\}_{i=1}^N$ is an {\bf unit norm frame}. 

\item $\{\langle x,\varphi_i\rangle \}_{i = 1}^N$ are called the {\bf frame coefficients} of the vector $x\in \mathcal{H}_M$ with respect to frame $\{\varphi_i\}_{i=1}^N$.

\item If the frame is unit norm and there is a constant d so that $|\langle \varphi_i, \varphi_j\rangle|=d$ for all $1\leq i \neq j \leq N$, then $\{\varphi_i\}_{i=1}^N$ is an {\bf equiangular frame}.

\item The {\bf frame operator S} of the frame is given by
\[ Sx = \sum_{i=1}^N\langle x,\varphi_i\rangle \varphi_i.\]

\end{enumerate}
\end{definition}

The frame operator is a particularly important object in the study of frame theory. Here we present an important result result about this operator. The proof can be found in \cite{CK}.
\begin{theorem}
Let $\{\varphi_{i}\}_{i=1}^N$ be a frame for $\HH_M$ with frame bounds $A$ and $B$ and frame operator $S$. $S$ is positive, invertible, and self adjoint. Moreover, the optimal frame bounds of $\{\varphi_i\}_{i=1}^N$ are given by $A=\lambda_{min}(S)$ and $B=\lambda_{max}(S)$ (the maximum and minimum eigenvalues of $S$). 
\end{theorem}
Now we make a simple observation about the frame bound of a unit norm tight frame.
\begin{proposition}
If $\{\varphi_i\}_{i=1}^N$ is a unit norm tight frame for $\mathcal{H}_M$ then the frame bound will be $A=N/M$.
\end{proposition}


\section{Estimating the number of non-zero frame coefficients}

In this section we will provide estimates on the number of 
indices for which the frame coefficients are non-zero. We will be looking for vectors producing the minimal number of non-zero frame coefficients, since there is always a dense set of vectors which have non-zero inner products with all the frame vectors.   For example, if $\{\varphi_i\}_{i=1}^N$ is a frame for $\HH_M$, then
$\varphi_i^{\perp}$ is a hyperplane for every $i\in [N]$.
If we choose 
\[ x \notin \cup_{i=1}^N \varphi_i^{\perp},\]
we have that 
\[ |\{i\in [N]:\langle x,\varphi_i\rangle \not= 0\}| = N.\]

In general, given a vector $x$ there may be only a few indices for which $\langle x,\varphi_i \rangle \not= 0$. For example, given $K$ copies 
of an orthonormal basis $\{e_i\}_{i=1}^{M}$ for $\HH_{M}$, say 
$\{e_{ij}\}_{1\leq i\leq M, \ 1\leq j\leq K},$
and choosing $x=e_{1}$ gives 
\begin{equation*}
\langle x, e_{ij}\rangle=0,
\quad\forall\,i\in\{2,\ldots,M\},\;j\in\{1,\ldots,K\}.
\end{equation*}
That is, we have only $K$ non-zero coefficients out of a total of 
$KM$. In other words, there are $KM/M$ non-zero coefficients. 
This is minimal with respect to this property as we demonstrate shortly.

\begin{theorem}
Let $\Phi=\{\varphi_i\}_{i=1}^{N}$ be a frame in $\HH_{M}$ with 
frame bounds $A,B$ and set 
$D:=\max\{\|\varphi_{i}\|^{2}:i\in [N]\}$. For any unit 
norm $x\in\HH_{M}$ define $J_x:=\{i\in [N]:\langle x,\varphi_i\rangle \not= 0\}$. Then
\begin{equation}\label{ENTFAA1.44}
|J_x|
\geq
\frac{A}{D}.
\end{equation}
So if $\Phi$ is a unit norm frame, then $|J_x| \ge A$ and
if it is a unit norm tight frame then $|J_x| \ge N/M$.

Moreover, if we have equality in \eqref{ENTFAA1.44} then
 the 
sub-collection of frame vectors $\{\varphi_{i}:i\in J_{x}\}$ 
spans a one-dimensional space.
\end{theorem}

\begin{proof}
Pick a unit norm $x\in\HH_{M}$ and set 
$J_{x}:=\{1\leq i\leq N:\langle x,\varphi_{i}\rangle\neq0\}$. Then
\begin{align}\label{ENTFAA1.45}
A
&=
A\|x\|^{2}
\leq
\sum\limits_{i=1}^{N}|\langle x,\varphi_{i}\rangle|^{2}
=
\sum\limits_{i\in J_{x}}|\langle x,\varphi_{i}\rangle|^{2}
\nonumber\\[2 pt]
&\leq
\sum\limits_{i\in J_{x}}\|x\|^{2}\|\varphi_{i}\|^{2}
=
\sum\limits_{i\in J_{x}}\|\varphi_{i}\|^2
\leq
D|J_{x}|.
\end{align}
Hence, $A/D$ is a lower bound for $|J_{x}|$ independent 
of each unit norm $x\in\HH_{M}$. It 
follows that \eqref{ENTFAA1.44} holds true.

Concerning the moreover part, if we have equality in 
\eqref{ENTFAA1.45}, then 
\[ |\langle x,\varphi_i\rangle|^2 = \|\varphi_i\|^2,\mbox{ for all 
} i \in J_x.\]
It follows that $\varphi_i = c_i x$, for some $|c_i|=1$, and
all $i\in J_x$.
\end{proof}
Note that the example preceding the theorem satisfies the minimal number of nonzero frame coefficients in the theorem. We may be interested in not only knowing when coefficients are nonzero, but when they are bounded away from zero as well. We provide lower bounds in the following theorem which can be viewed as a generalization of the previous.

\begin{theorem}
Let $C\in(0,1)$, let $\Phi=\{\varphi_{i}\}_{i=1}^{N}$ be a frame for 
$\HH_{M}$ with frame bounds $A,B$ and set 
$D:=\max\{\|\varphi_{i}\|^{2}:i\in [N]\}$. For a unit 
norm $x\in\HH_{M}$, define $\displaystyle K_x:=\{i\in [N]:|\langle x,\varphi_{i}\rangle|^{2}
>(CA)/N\}$. We have
\begin{equation}\label{ENTFAA1.52}
|K_x|\geq(1-C)\frac{A}{D}.
\end{equation}
In particular, if $\Phi$ is a unit norm tight frame, then 
\begin{equation*}
|K_x| \geq(1-C)\frac{N}{M}.
\end{equation*}
\end{theorem}

\begin{proof}
For unit norm $x\in\HH_{M}$, we have 
\begin{equation*}
\sum\limits_{i\in K_{x}^{c}}|\langle x,\varphi_{i}\rangle|^{2}
\leq\frac{CA}{N}|K_{x}^{c}|
\qquad\mbox{and}\qquad
\sum\limits_{i\in K_{x}}
|\langle x,\varphi_{i}\rangle|^{2}\leq D|K_{x}|.
\end{equation*}
Hence,
\begin{align*}
A
&=
A\|x\|^{2}
\leq
\sum\limits_{i=1}^{N}|\langle x,\varphi_{i}\rangle|^{2}
=
\sum\limits_{i\in K_{x}}|\langle x,\varphi_{i}\rangle|^{2}
+\sum\limits_{i\in K_{x}^{c}}|\langle x,\varphi_{i}\rangle|^{2}
\nonumber\\[2 pt]
&\leq
D|K_{x}|+\frac{CA}{N}|K_{x}^{c}|
=
D|K_{x}|+\frac{CA}{N}(N-|K_{x}|)
\nonumber\\[2 pt]
&=
\Big(D-\frac{CA}{N}\Big)|K_{x}|+CA.
\end{align*}
It follows that
\begin{equation*}
A(1-C)\leq\Big(D-\frac{CA}{N}\Big)|K_{x}|.
\end{equation*}
By the definition of $K_{x}$ and the Cauchy-Schwarz inequality, we 
may deduce that $D>(CA)/N$.  Consequently 
\begin{equation}\label{ENTFAA1.56}
|K_{x}|\geq(1-C)\frac{A}{D-(CA)/N}\geq(1-C)\frac{A}{D}.
\end{equation}
Since the lower bound in \eqref{ENTFAA1.56} is independent of the 
unit norm $x\in\HH_{M}$, the validity of \eqref{ENTFAA1.52} 
follows.
\end{proof}

\section{The Distribution of the Frame Coefficients}

In this section we will classify the distribution of the frame coefficients using {\it majorization}. Majorization has an interesting theory in itself of which we will need only a small portion. The interested reader should refer to \cite{MAJORIZATION} for a detailed study of majorization inequalities and their applications.

\begin{definition}
Let $a=(a_i)_{i=1}^N$ and $b=(b_i)_{i=1}^N$ be vectors with non-negative, non-increasing components.
\begin{enumerate}[(i)]
	\item We say $a$ {\bf majorizes} $b$ and write $a\succ b$ if 
	\begin{equation*}
\sum\limits_{i=1}^{k}a_{i}\geq\sum\limits_{i=1}^{k}b_{i},
\;\;\forall\,k\in\{1,2,\ldots,N\}
\quad\mbox{and}\quad
\sum\limits_{i=1}^{N}a_{i}=\sum\limits_{i=1}^{N}b_{i}.
\end{equation*}
	\item We say $a$ {\bf weakly majorizes} $b$ and write $a\succ_W b$ if 
	
\begin{equation*}
\sum\limits_{i=1}^{k}a_{i}\geq\sum\limits_{i=1}^{k}b_{i},
\;\;\forall\,k\in\{1,2,\ldots,N\}.
\end{equation*}
\end{enumerate}

If the two vectors do not have the same number of terms we agree to pad the shorter vector with zeros so that they have the same length.

\end{definition}

It will be important for our work to understand what happens when partial sums of the majorization vectors are equal.

\begin{proposition}\label{ENTFAA1.6}
Let $a:=(a_{i})_{i=1}^{N}\in\mathbb{R}^{N}$ be a non-negative, 
non-increasing sequence of numbers, let 
$b:=(A,A,\ldots,A)\in\mathbb{R}^{N}$ where $A>0$, and assume 
$a\succ_{W}b$. If there exists $m\in\{1,\ldots,N\}$ so that 
$\sum\limits_{i=1}^{m}a_{i}=mA$, then $a_{i}=A$ for every 
$i\in\{m+1,\ldots,N\}$.
\end{proposition}

\begin{proof}
Assume there exists $m\in\{1,\ldots,N\}$ so that 
$\sum\limits_{i=1}^{m}a_{i}=mA$. Then $a_{m}\leq A$ and, hence, 
$a_{i}\leq A$ for every $i\in\{m+1,\ldots,N\}$. Assume for the 
moment that $a_{i}<A$ for some $i\in\{m+1,\ldots,N\}$, then 
\begin{align*}
NA
&=
\sum\limits_{i=1}^{N}A
\leq
\sum\limits_{i=1}^{N}a_{i}
=
\sum\limits_{i=1}^{m}a_{i}
+
\sum\limits_{i=m+1}^{N}a_{i}
\nonumber\\[2 pt]
&<
mA+(N-m)A
=
NA,
\end{align*}
which is a contradiction. Thus, $a_{i}=A$ for every 
$i\in\{m+1,\ldots,N\}$. 
\end{proof}

The following proposition establishes a relationship between the largest eigenvalue of the frame operator 
 and the number of frame 
coefficients of modulus one.

\begin{proposition}
Let $\{\varphi_i\}_{i=1}^{N}$ be a unit norm frame in $\HH_{M}$ 
whose frame operator has eigenvalues 
$\lambda_{1}\geq\dots\geq\lambda_{M}$ and let 
$x\in\HH_{M}$ be unit norm. Define the vector $a:=(a_{i})_{i=1}^{N}\in\mathbb{R}^{N}$ to be 
$(|\langle x,\varphi_{i}\rangle|^{2})_{i=1}^{N}$ arranged in 
non-increasing order. Then $a_{i}=1$ for at most 
$\lfloor\lambda_{1}\rfloor$ indices $i$ (where 
$\lfloor\,\cdot\,\rfloor$ represents the least integer function).
\end{proposition}

\begin{proof}
By definition, 
$\sum\limits_{i=1}^{N}|\langle x,\varphi_i\rangle|^{2}
\leq\lambda_{1}\|x\|^2=\lambda_{1}$. 
Suppose for the moment that $a_{i}=1$ for every 
$i\in\{1,\ldots,\lfloor\lambda_{1}\rfloor+1\}$, then 
\begin{align*}
\lambda_{1}
&\geq
\sum\limits_{i=1}^{N}|\langle x,\varphi_i\rangle|^{2}
=
\sum\limits_{i=1}^{\lfloor\lambda_{1}\rfloor+1}
|\langle x,\varphi_i\rangle|^{2}
+
\sum\limits_{i=\lfloor\lambda_{1}\rfloor+2}^{N}
|\langle x,\varphi_i\rangle|^{2}
\nonumber\\[2 pt]
&=
(\lfloor\lambda_{1}\rfloor+1)
+
\sum\limits_{i=\lfloor\lambda_{1}\rfloor+2}^{N}
|\langle x,\varphi_i\rangle|^{2}
>
\lambda_{1}, 
\end{align*}
which is a contradiction.
\end{proof}

We are ready to give one of the main majorization results
for frame coefficients of unit norm vectors.

\begin{theorem}
Let $\{\varphi_i\}_{i=1}^{N}$ be a frame for $\HH_{M}$ with frame 
bounds $A\leq B$ and let 
\begin{equation*}
b:=\left(\frac{A}{N},\frac{A}{N},\ldots,\frac{A}{N}\right)
\in\mathbb{R}^N.
\end{equation*}
For any unit norm $x\in\HH_{M}$, define $a:=(a_i)_{i=1}^{N}\in\mathbb{R}^{N}$ to be 
$(|\langle x,\varphi_i\rangle|^2)_{i=1}^{N}$ arranged in 
non-increasing order. Then $a\succ_W b$.
In particular, $a_{1}\geq A/N$.

Moreover, if there exists an $m\in\{1,\ldots N\}$ such that 
$\sum\limits_{i=1}^{m}a_{i}=m(A/N)$ then 
$a_{i}=A/N$ for all $i\in\{m+1,\ldots,N\}$.
\end{theorem}

\begin{proof}
We proceed by way of contradiction.  If there is an 
$m\in\{1,\ldots,N\}$ so that 
\begin{equation*}
\sum\limits_{i=1}^{m}a_{i}
<
\sum_{i=1}^{m}\frac{A}{N}
=
m\frac{A}{N},
\end{equation*}
then $a_{m}<A/N$. Furthermore, it follows that 
\begin{equation*}
a_{i}\leq a_{m},\;\;\forall\,i\in\{m,\ldots,N\},
\end{equation*}
and we may write for any unit vector $x\in\HH_{M}$ 
\begin{align*}
A
&=
A\|x\|^{2}
\leq
\sum\limits_{i=1}^{N}a_{i}
=
\sum\limits_{i=1}^{m}a_{i}+\sum\limits_{i=m+1}^{N}a_{i}
\nonumber\\[2 pt]
&<
m\frac{A}{N}+(N-m)a_{m}
<
m\frac{A}{N}+(N-m)\frac{A}{N}
\nonumber\\[2 pt]
&=
A,
\end{align*}
a contradiction.

{\it The moreover part} follows from 
Proposition~\ref{ENTFAA1.6}.
\end{proof}

\begin{corollary}
Let $\{\varphi_i\}_{i=1}^N$ be a unit norm tight frame in 
$\HH_{M}$. For 
any unit norm $x\in\HH_{M}$, if 
$a:=(a_i)_{i=1}^{N}\in\mathbb{R}^{N}$ is
$(|\langle x,\varphi_{i}\rangle|^2)_{i=1}^{N}$ arranged in 
non-increasing order, and we set 
\begin{equation*}
b:=\Big(\frac{1}{M},\frac{1}{M},\ldots,\frac{1}{M}\Big)
\in\mathbb{R}^N.
\end{equation*}
Then $a\succ_{W}b$.
\end{corollary}

Given a frame $\{\varphi_{i}\}_{i=1}^{N}$ and a unit vector $x \in \HH_{M}$, if $x$ is orthogonal to a number of frame vectors, we would expect the nonzero coefficients to 
start to grow, and they indeed do as we shall see in the following 
proposition. 

\begin{proposition}
Let $\{\varphi_{i}\}_{i=1}^{N}$ be a frame for $\HH_{M}$ with 
frame bounds $A,B$. For a given unit norm $x\in \HH_M$, define $I:=\{i\in [N]:x\perp \phi_i\}$ and let $|I|=K$.
Finally, let
$a:=(a_{i})_{i=1}^{N}\in\mathbb{R}^{N}$ be 
$(|\langle x,\varphi_{i}\rangle|^{2})_{i=1}^{N}$ arranged in 
non-increasing order and set
\begin{equation*}
b
:=
\Big(
\underbrace{\frac{A}{N-K},\ldots,\frac{A}{N-K}}
_{N-K\;\mbox{terms}},
\underbrace{0,\ldots,0}_{K\;\mbox{terms}}\Big)
\in\mathbb{R}^{N}.
\end{equation*}
Then $a\succ_{W}b$.
\end{proposition}

\begin{proof}
Without loss of generality, assume $I=\{N-K+1,\ldots,N\}$ 
(if not, apply a permutation to the indices $\{1,\ldots,N\}$) and thus, ${I^{c}=\{1,\ldots,N-K\}}$. 
Furthermore, since $(a_{i})_{i=1}^{N}$ is in non-increasing 
order and $x\perp\varphi_{i}$ for each $i\in I$, then $a_{i}=0$ 
for all $i\in I$. Being that $a_{i}\geq0$ for all 
$i\in\{1,\ldots,N\}$, it suffices to show 
\[\sum\limits_{i=1}^{m}a_{i}\geq
\sum\limits_{i=1}^{m}\frac{A}{N-K}\] 
for all $m\in\{1,\ldots,N-K\}$. We proceed by way of 
contradiction.  If there exists $m\in\{1,\ldots,N-K\}$ so that
\begin{equation*}
\sum\limits_{i=1}^{m}a_{i}
<
\sum\limits_{i=1}^{m}\frac{A}{N-K}
=
m\frac{A}{N-K},
\end{equation*}
then $\displaystyle a_{m}<\frac{A}{N-K}$. Moreover, it follows that 
\begin{equation*}
a_{i}\leq a_{m},\;\;\forall\,i\in\{m,\ldots,N-K\},
\end{equation*}
and we may write for any unit vector $x\in\HH_{M}$
\begin{align*}
A
&=
A\|x\|^{2}
\leq
\sum\limits_{i=1}^{N}a_{i}
=
\sum\limits_{i\in I^{c}}a_{i}+\sum\limits_{i\in I}a_{i}
=
\sum\limits_{i\in I^{c}}a_{i}
=
\sum\limits_{i=1}^{m}a_{i}+\sum\limits_{i=m+1}^{N-K}a_{i}
\nonumber\\[2 pt]
&<
\sum\limits_{i=1}^{m}\frac{A}{N-k}+\sum\limits_{i=m+1}^{N-K}a_{m}
=
m\frac{A}{N-K}+(N-K-m)a_{m}
\nonumber\\[2 pt]
&<
m\frac{A}{N-K}+(N-K-m)\frac{A}{N-K}
=
A,
\end{align*}
a contradiction.
\end{proof}


\section{Products of frame coefficients}

In this section, we will investigate the following: given a frame $\{\varphi_i\}_{i=1}^N$ for $\HH_M$, what can be said about the quantity
\begin{equation}\label{ENTFAA2.2.1}
\min\limits_{\|x\|=1=\|y\|}\sum\limits_{i=1}^{N}
|\langle x,\varphi_{i}\rangle||\langle y,\varphi_{i}\rangle|?
\end{equation}

We start with a few basic observations.

\begin{observation}
If $M\leq 2N-2$ we can always choose $x$ and $y$ such that 
$x\perp\varphi_{i}$ for all $i\in I_{x}\subset\{1,\dots,N\}$, 
$|I_{x}|\le N-1$, and $y\perp\varphi_{i}$ for all 
$i\in I_{y}\subset\{1,\dots,N\}$, $|I_{y}|\le N-1$, so that 
$I_{x}\cap I_{y}=\emptyset$ making the quantity \eqref{ENTFAA2.2.1} zero.
\end{observation}

\begin{observation}
If we have a Hilbert space 
$\HH_{M_{1}}\oplus\HH_{M_{2}}$ and vectors 
$\{\varphi_{i}\}_{i\in I}$ in $\HH_{M_{1}}$ and 
$\{\varphi_{i}\}_{i\in J}$ in $\HH_{M_{2}}$, then by choosing 
$x\in\HH_{M_{1}}$ and $y\in\HH_{M_{2}}$ unit norm, we have
\begin{equation*}
\sum\limits_{i\in I\cup J}
|\langle x,\varphi_{i}\rangle||\langle y,\varphi_{i}\rangle|=0.
\end{equation*}
\end{observation}
In light of the above observations, nothing can be said about \eqref{ENTFAA2.2.1} \emph{in general}. However, there are cases where we can produce meaningful results.
We proceed analogously to the preceding sections in that we begin by estimating the number of non-zero summands. That is, we wish to know for how many $i \in [N]$ do we have
$|\langle x,\varphi_i\rangle||\langle y,\varphi_i\rangle|
\not= 0$. 
\begin{definition}
	Let $\Phi=\{\varphi_{i}\}_{i=1}^N$ be a frame for $\HH_M$. We say that $\Phi$
	\begin{enumerate}[(i)]
		\item is {\bf full spark} if for every $I\subset [N]$, $\{\varphi_i\}_{i\in I}$, with $|I|=M$ spans $\HH_M$,
		\item has the {\bf complement property} if for every $I\subset [N]$, with $|I|=M$ either
$\{\varphi_i\}_{i\in I}$ or $\{\varphi_i\}_{i\in I^c}$ 
spans $\HH_M$.
	\end{enumerate}
\end{definition}

\begin{proposition}\label{ENTFAA2.6}
Let $\{\varphi_{i}\}_{i=1}^{N}$ be a frame for $\HH_{M}$ with 
$N\geq2M-1$. The following are equivalent.

\begin{enumerate}[(i)]
\item $\{\varphi_{i}\}_{i=1}^{N}$ has the complement property.
\item For every $x,y\in\HH_{M}$ we have
$\sum\limits_{i=1}^{N}
|\langle x,\varphi_{i}\rangle||\langle y,\varphi_{i}\rangle|
\neq0$.
\end{enumerate}
Moreover, if $\{\varphi_{i}\}_{i=1}^{N}$ is full spark, then 
\begin{equation*}
\big|\{i:|\langle x,\varphi_{i}\rangle|
|\langle y,\varphi_{i}\rangle|\neq0\}\big|
\geq
N-(2M-2).
\end{equation*}
\end{proposition}

\begin{proof}\hfill

(i)$\Rightarrow$(ii): We will prove the contrapositive. Suppose there exist nonzero $x,y\in\HH_{M}$ such that 
$\sum\limits_{i=1}^{N}|\langle x,\varphi_{i}\rangle|
|\langle y,\varphi_{i}\rangle|=0$ and define
\begin{equation*}
\begin{array}{l}
I:=\{1\leq i\leq N:\langle x,\varphi_{i}\rangle=0\},
\\[4 pt]
J:=\{1\leq i\leq N:\langle y,\varphi_{i}\rangle=0\}.
\end{array}
\end{equation*}
\noindent\underline{Case 1:}
If $I\cap J=\emptyset$, then $J=I^{c}$. Furthermore, 
$x\perp\varphi_{i}$ for all $i \in I$ and thus 
${\rm span}(\{\varphi_{i}\}_{i\in I})\neq\HH_{M}$. Also, 
$y\perp\varphi_{i}$ for all $i\in J=I^{c}$ and thus 
${\rm span}(\{\varphi_{i}\}_{i \in I^{c}})\neq\HH_{M}$. Therefore 
there exists a partition $I,I^{c}$ of 
$\{\varphi_{i}\}_{i=1}^{N}$ for which neither set spans 
$\HH_{M}$. Hence, $\{\varphi_{i}\}_{i=1}^{N}$ fails the complement 
property.

\noindent\underline{Case 2:}
If $I\cap J\neq\emptyset$, then 
$\big(J\setminus(I\cap J)\big)=I^{c}$. Furthermore, 
$x\perp\varphi_{i}$ for all $i\in I$ and thus 
${\rm span}(\{\varphi_{i}\}_{i \in I})\neq\HH_{M}$. Also, 
$y\perp\varphi_{i}$ for all 
$i\in\big(J\setminus(I\cap J)\big)=I^{c}$ and thus 
${\rm span}(\{\varphi_{i}\}_{i\in I^{c}})\neq\HH_{M}$. Therefore 
there exists a partition $I,I^{c}$ of $\{\varphi_{i}\}_{i=1}^{N}$ 
for which neither set spans $\HH_{M}$. Thus, 
$\{\varphi_{i}\}_{i=1}^{N}$ fails the complement property.

\noindent (ii)$\Longrightarrow$(i): Again by contrapositive. Suppose $\{\varphi_{i}\}_{i=1}^{N}$ fails the 
complement property. This implies there exists a partition 
$I,I^{c}$ of $\{1,\dots,N\}$ such that 
${\rm span}(\{\varphi_{i}\}_{i\in I})\neq\HH_{M}$ and 
${\rm span}(\{\varphi_{i}\}_{i\in I^{c}})\neq\HH_{M}$. Thus, 
there exist $x,y\in\HH_{M}$ with the property that 
$x\perp\{\varphi_{i}\}_{i\in I}$ and 
$y\perp\{\varphi_{i}\}_{i\in I^{c}}$. Therefore 
$\sum\limits_{i=1}^{N}|\langle x,\varphi_{i}\rangle|
|\langle y,\varphi_{i}\rangle|=0$, 
as wanted.

For the moreover part, by full spark we have 
\begin{equation*}
\big|\{i:|\langle x,\varphi_{i}\rangle|=0\}\big|\leq M
\qquad\mbox{and}\qquad
\big|\{i:|\langle y,\varphi_{i}\rangle|=0\}\big|\leq M,
\end{equation*}
as wanted.
\end{proof}

We digress for a moment to point out an interesting connection to the problem of phaseless signal reconstruction. Introduced in \cite{PHASELESS}, Balan, Casazza and Ediden proved necessary and sufficient conditions for a frame to do phaseless reconstruction. That is, if we define the map 
\[A:\RR^M/\{\pm1\}\to \RR^N_{\geq 0}\]
by
\[A:x\mapsto (|\ip{x}{\varphi_i}|^2),\]
under what circumstances is this map injective? The authors demonstrate that $A$ is injective if and only if for every subset $I\subset[N]$, we have that $\{\varphi_i\}_{i\in I}$ or $\{\varphi_i\}_{i\in I^c}$ spans. This is equivalent to (ii) in Proposition~\ref{ENTFAA2.6}.

Next, we give an upper bound on the sum in equation~\eqref{ENTFAA2.2.1}.

\begin{lemma}\label{ENTFAA3.1}
Let $\{\varphi_{i}\}_{i=1}^{N}$ be a frame in $\HH_{M}$ with 
frame bounds $A,B$. Then 
\begin{equation*}
\sup\Big\{\sum\limits_{i=1}^{N}
|\langle x,\varphi_{i}\rangle|
|\langle y,\varphi_{i}\rangle|:x,y\in\HH_{M},\|x\|=1=\|y\|\Big\}
\leq
B.
\end{equation*}
\end{lemma}

\begin{proof}
Let $x,y\in\HH_{M}$ be unit norm. Invoking H\"older's inequality 
gives
\begin{align*}
\sum\limits_{i=1}^{N}
|\langle x,\varphi_{i}\rangle|
|\langle y,\varphi_{i}\rangle|
&\leq
\Big(\sum\limits_{i=1}^{N}
|\langle x,\varphi_{i}\rangle|^{2}\Big)^{1/2}
\Big(\sum\limits_{i=1}^{N}
|\langle y,\varphi_{i}\rangle|^{2}\Big)^{1/2}
\nonumber\\[2 pt]
&\leq
B^{1/2}B^{1/2}
=
B.
\end{align*}
\end{proof}

The following proposition concerns another upper 
bound estimate for the summation in \eqref{ENTFAA2.2.1}. It should 
be noted that the proposition is neither a stronger nor weaker version of Lemma~\ref{ENTFAA3.1}. This can be seen by 
making particular choices of $N,M$ which make the bounds in the lemma better than the bounds in the proposition and vice versa. 

\begin{proposition}\label{ENTFAA3.32}
Let $N\geq 2M-1$ and let $\{\varphi_{i}\}_{i=1}^{N}$ be a frame 
for $\HH_{M}$ with frame bounds $A,B$. Then
\begin{equation*}
\sup\limits_{\|x\|=1=\|y\|}
\Big\{\sum\limits_{i=1}^{N}
|\langle x,\varphi_{i}\rangle||\langle y,\varphi_{i}\rangle|\Big\}
\leq
(N-2M+2)\sqrt{\frac{B}{M}}.
\end{equation*}
\end{proposition}

\begin{proof}
Pick unit vectors $x,y\in\HH_{M}$ so that $x\perp\varphi_{i}$ 
for all $i\in\{1,\ldots,M-1\}$ and $y\perp\varphi_{i}$ for all 
$i\in\{N-M+2,\ldots,N\}$. It follows that 
\begin{equation*}
B
=
B\|x\|^{2}
\geq
\sum_{i=1}^{N}|\langle x,\varphi_{i}\rangle|^2
=
\sum_{i=M}^{N}|\langle x,\varphi_{i}\rangle|^2.
\end{equation*}
Suppose momentarily that 
$|\langle x,\varphi_{i}\rangle|^{2}>\frac{B}{M}$ for all 
$i\in\{M,\ldots,N\}$. Then 
\begin{equation}\label{ENTFAA3.35}
B
\geq
\sum\limits_{i=M}^{N}|\langle x,\varphi_{i}\rangle|^{2}
>
(N-M+1)\,\frac{B}{M}
\geq
M\frac{B}{M}
=
B,
\end{equation}
which is a contradiction. Hence, there exists 
$i_{1}\in\{M,\ldots,N\}$ such that 
$|\langle x,\varphi_{i_{1}}\rangle|^{2}\leq\frac{B}{M}$. If 
the third inequality in \eqref{ENTFAA3.35} is strict, then using 
a similar proof by contradiction $($as in \eqref{ENTFAA3.35}$)$ 
at most $N-2M+1$ more times, we obtain there exist indices 
$i_{2},\ldots,i_{N-2M+2}$ all not equal to $i_{1}$ so that 
$|\langle x,\varphi_{i_{j}}\rangle|^{2}\leq\frac{B}{M}$ for 
every $j\in\{1,\ldots,N-2M+2\}$. Without loss of generality, we 
may assume $i_{1}=M, i_{2}=M+1,\ldots,i_{N-2M+2}=N-M+1$. Then
\begin{align}\label{ENTFAA3.36}
\sum\limits_{i=1}^{N}
|\langle x,\varphi_{i}\rangle||\langle y,\varphi_{i}\rangle|
&=
\sum\limits_{i=M}^{N-M+1}
|\langle x,\varphi_{i}\rangle||\langle y,\varphi_{i}\rangle|
\leq
\sum\limits_{i=M}^{N-M+1}
|\langle x,\varphi_{i}\rangle|
\nonumber\\[2 pt]
&\leq
(N-2M+2)\sqrt{\frac{B}{M}}.
\end{align}
Taking the supremum over all $x,y\in\HH_{M}$ so that 
$\|x\|=1=\|y\|$ in \eqref{ENTFAA3.36} gives the desired result.
\end{proof}

\begin{remark} Now we compare the bounds in Lemma~\ref{ENTFAA3.1} with Proposition~\ref{ENTFAA3.32}. Let $N\geq 2M-1$ and let $\{\varphi_i\}_{i=1}^N$ be a unit norm tight frame for $\mathcal{H}_M$. If $N-\sqrt{N}+2 \geq 2M$ then $(N-2M+2)\frac{\sqrt{N}}{M} \geq \frac {N}{M}$ and hence $\frac{N}{M}$ is a better upper bound for 
\begin{equation*}
\sup\limits_{x,y\in\HH_{M},\;\|x\|=1=\|y\|}
\Big\{\sum\limits_{i=1}^{N}
|\langle x,\varphi_{i}\rangle||\langle y,\varphi_{i}\rangle|\Big\}.
\end{equation*}
Moreover, when $N-\sqrt{N}+2 < 2M$ then $(N-2M+2)\frac{\sqrt{N}}{M}< \frac{M}{N}$ and hence $(N-2M+2)\frac{\sqrt{N}}{M} $ is a better bound for this supremum. 
\end{remark}

We now consider lower bounds for the summation in 
\eqref{ENTFAA2.2.1} in the case of tight frames. The next proposition demonstrates that this quantity is at 
least the product of the tight frame bound and the inner product 
of the unit vectors.  

\begin{proposition}
If $\{\varphi_{i}\}_{i=1}^{N}$ is a unit norm tight frame in 
$\HH_{M}$, then for any unit norm $x,y\in\HH_{M}$ we have 
\begin{equation*}
\frac{N}{M}|\langle x,y\rangle|
\leq
\sum\limits_{i=1}^{N}
|\langle x,\varphi_{i}\rangle||\langle y,\varphi_{i}\rangle|.
\end{equation*}
\end{proposition}

\begin{proof}
Let $x,y\in\HH_{M}$ be unit norm and consider
\begin{align*}
\frac{N}{M}|\langle x,y\rangle|
&=
\Big|\big\langle x,\frac{N}{M}y\big\rangle\Big|
=
\Big|\Big\langle x,\sum\limits_{i=1}^{N}
\langle y,\varphi_{i}\rangle\varphi_{i}\Big\rangle\Big|
=
\Big|\sum\limits_{i=1}^{N}
\langle x,\varphi_{i}\rangle
\overline{\langle y,\varphi_{i}\rangle}\Big|
\nonumber\\[2 pt]
&\leq
\sum\limits_{i=1}^{N}
|\langle x,\varphi_{i}\rangle||\langle y,\varphi_{i}\rangle|.
\end{align*}
\end{proof}

Continuing our investigation for lower bounds in the summation 
in equation~\eqref{ENTFAA2.2.1}, we make a slight modification in the 
following lemma. We will consider that $y$ is a \emph{fixed} element of our frame and allow $x$ to vary.

\begin{lemma}
Let $\{\varphi_{i}\}_{i=1}^{N}$ be an equiangular unit norm tight 
frame in $\HH_{M}$ and fix $j\in\{1,\ldots,N\}$. Then 
\begin{equation*}
\inf\Big\{\sum\limits_{i=1}^{N}
|\langle\varphi_{j},\varphi_{i}\rangle|
|\langle x,\varphi_{i}\rangle|:
x\in\HH_{M},\|x\|=1\Big\}
\geq
\frac{N}{M}\sqrt{\frac{N-M}{M(N-1)}}.
\end{equation*}
In particular, when 
\begin{enumerate}[(i)]
\item $N=2M$, we get 
\begin{equation*}
\inf\Big\{\sum\limits_{i=1}^{N}
|\langle\varphi_{j},\varphi_{i}\rangle|
|\langle x,\varphi_{i}\rangle|:
x\in\HH_{M},\|x\|=1\Big\}
\geq
\frac{2}{\sqrt{2M-1}};
\end{equation*}

\item $N=\frac{M(M+1)}{2}$, we get
\begin{equation*}
\inf
\Big\{\sum\limits_{i=1}^{N}
|\langle\varphi_{j},\varphi_{i}\rangle|
|\langle x,\varphi_{i}\rangle|:
x\in\HH_{M},\|x\|=1\Big\}
\geq
\frac{M+1}{2\sqrt{M+2}}.
\end{equation*}
\end{enumerate}
\end{lemma}

\begin{proof}
Fix a unit norm $x\in\HH_{M}$ and fix $j\in\{1,\ldots,N\}$. Next, 
define $I:=\{1,\ldots,N\}\setminus\{j\}$ and 
$c:=\sqrt{\frac{N-M}{M(N-1)}}$. Consider 
\begin{align}
\sum\limits_{i=1}^{N}
|\langle\varphi_{j},\varphi_{i}\rangle|
|\langle x,\varphi_{i}\rangle|
&=
\sum\limits_{i\in I}
|\langle\varphi_{j},\varphi_{i}\rangle|
|\langle x,\varphi_{i}\rangle|
+
|\langle\varphi_{j},\varphi_{j}\rangle|
|\langle x,\varphi_{j}\rangle|
\nonumber\\[2 pt]
&=
c\sum\limits_{i\in I}|\langle x,\varphi_{i}\rangle|
+
|\langle x,\varphi_{j}\rangle|
\nonumber\\[2 pt]
&=
c\sum\limits_{i=1}^{N}|\langle x,\varphi_{i}\rangle|
+
(1-c)|\langle x,\varphi_{j}\rangle|
\nonumber\\[ 2 pt]
&\geq
c\sum\limits_{i=1}^{N}|\langle x,\varphi_{i}\rangle|
\geq
c\sum\limits_{i=1}^{N}|\langle x,\varphi_{i}\rangle|^{2}
=
c\,\frac{N}{M}.
\end{align}
This is a lower bound for the set 
$\Big\{\sum\limits_{i=1}^{N}
|\langle\varphi_{j},\varphi_{i}\rangle|
|\langle x,\varphi_{i}\rangle|:x\in\HH_{M},\|x\|=1\Big\}$ 
which is independent of the unit norm $x\in\HH_M$. 
\end{proof}





\section{Distance between vectors and frame vectors}

In this section, we will give good estimates of the squared sums of the distances between a vector and the frame vectors.  We wil discover some surprising uniformities for the equiangular case.

The following lemma establishes a relationship between the 
coefficients of a collection of vectors and the norms and inner 
products of these vectors. While interesting in its own right, 
this lemma serves as a tool in 
Proposition~\ref{ENTFAA2.26}.

\begin{lemma}\label{ENTFAA2.21}
Let $\{\varphi_{i}\}_{i=1}^{N}$ be a collection of vectors in 
$\mathbb{R}^{M}$ whose components are given as 
$\varphi_{1}=(\varphi_{11},\ldots,\varphi_{1M}),\ldots,
\varphi_{N}=(\varphi_{N1},\ldots,\varphi_{NM})$. 
Then
\begin{align}\label{ENTFAA2.23}
\sum\limits_{j=1}^{M}
\Big(\sum\limits_{i=1}^{N}\varphi_{ij}\Big)^{2}
&=
\sum\limits_{i=1}^{N}\sum\limits_{j=1}^{M}\varphi_{ij}
\Big(\sum\limits_{k=1}^{N}\varphi_{kj}\Big)
\nonumber\\[2 pt]
&=
\sum\limits_{i=1}^{N}\|\varphi_{i}\|^{2}+
2\sum\limits_{1\leq i<k\leq N}
\langle\varphi_{i},\varphi_{k}\rangle.
\end{align}
\end{lemma}

\begin{proof}
Concerning the first equality, note that 
\begin{align*}
\sum\limits_{i=1}^{N}\sum\limits_{j=1}^{M}\varphi_{ij}
\Big(\sum\limits_{k=1}^{N}\varphi_{kj}\Big)
&=
\sum\limits_{j=1}^{M}\sum\limits_{i=1}^{N}\varphi_{ij}
\Big(\sum\limits_{k=1}^{N}\varphi_{kj}\Big)
=
\sum\limits_{j=1}^{M}\Big(\sum\limits_{k=1}^{N}\varphi_{kj}\Big)
\sum\limits_{i=1}^{N}\varphi_{ij}
\nonumber\\[2 pt]
&=
\sum\limits_{j=1}^{M}
\Big(\sum\limits_{i=1}^{N}\varphi_{ij}\Big)^{2}.
\end{align*}
For the last equality in \eqref{ENTFAA2.23}, consider
\begin{align*}
\sum\limits_{j=1}^{M}
\Big(\sum\limits_{i=1}^{N}\varphi_{ij}\Big)^{2}
&=
\sum\limits_{j=1}^{M}
\Big(\sum\limits_{i=1}^{N}\varphi_{ij}^{2}+
2\sum\limits_{1\leq i<k\leq N}\varphi_{ij}\varphi_{kj}\Big)
\nonumber\\[2 pt]
&=
\sum\limits_{i=1}^{N}\sum\limits_{j=1}^{M}\varphi_{ij}^{2}+
2\sum\limits_{1\leq i<k\leq N}\sum\limits_{j=1}^{M}
\varphi_{ij}\varphi_{kj}
\nonumber\\[2 pt]
&=
\sum\limits_{i=1}^{N}\|\varphi_{i}\|^{2}+
2\sum\limits_{1\leq i<k\leq N}
\langle\varphi_{i},\varphi_{k}\rangle,
\end{align*}
as desired.
\end{proof}

The proposition below establishes upper and lower bounds for the 
sum of the squares of the distances between any unit vector and 
the frame vectors in an equiangular tight frame in 
terms of the dimension and the modulus of the inner product between vectors. 
Recall from \cite{CK} that if $\{\varphi_{i}\}_{i=1}^{N}$ is an 
equiangular tight frame for $\HH_{M}$, then for $i\neq j$ we have 
\begin{equation}\label{ENTFAA2.20}
|\langle\varphi_{i},\varphi_{j}\rangle|^{2}=\frac{N-M}{M(N-1)}.
\end{equation}

\begin{proposition}\label{ENTFAA2.26}
Let $\{\varphi_{i}\}_{i=1}^{N}$ be an equiangular tight 
frame in $\mathbb{R}^{M}$. Then for any unit norm 
$x\in\mathbb{R}^{M}$, we have 
\begin{align*}
2\big(N-\sqrt{N[1+(N-1)c]}\big)
&\leq
\sum\limits_{i={1}}^{N}\lVert x-\varphi_{i}\rVert^{2}
\nonumber\\[2 pt]
&\leq
2\big(N+\sqrt{N[1+(N-1)c]}\big),
\end{align*}
where $c:=\sqrt{\frac{N-M}{M(N-1)}}=\lvert\langle\varphi_{i},\varphi_{j}\rangle\rvert,$ for all $i\neq j$.
\end{proposition}

\begin{proof}
Fix a unit norm $x\in\mathbb{R}^{M}$ and write the vectors 
$x,\varphi_{1},\ldots,\varphi_{N}\in\mathbb{R}^{M}$ in terms of 
their components; that is, write $x=(x_{1},\ldots,x_{M})$ and $
\varphi_{i}=(\varphi_{i1},\ldots,\varphi_{iM})$ for each $i$.

Next, define the functions 
$f,g:\mathbb{R}^{M}\longrightarrow\mathbb{R}$ by 
\begin{equation*}
f(y):=\sum\limits_{i=1}^{N}\lVert y-\varphi_{i}\rVert^{2}
\qquad\mbox{and}\qquad
g(y):=\|y\|^{2}.
\end{equation*}
Substituting our unit norm $x\in\mathbb{R}^{M}$ into the functions 
$f,g$ give 
\begin{equation}\label{ENTFAA2.30}
\begin{array}{c}
f(x)
=
\sum\limits_{i=1}^{N}\big(\|x\|^{2}+\|\varphi_{i}\|^{2}
-2\langle x,\varphi_{i}\rangle\big)
=
2N-2\sum\limits_{i=1}^{N}\sum\limits_{j=1}^{M}x_{j}\varphi_{ij};
\\[2 pt]
g(x)=\sum\limits_{i=1}^{M}x_{i}^{2}=1.
\end{array}
\end{equation}
At this stage, the main idea is to use the method of Lagrange multipliers on 
the function $f(x)$ subject to the constraint function $g(x)$ to 
identify any absolute extrema. To this end, we calculate 
$(\nabla f)(x)$ and $(\nabla g)(x)$ as 
\begin{equation*}
\begin{array}{c}
(\nabla f)(x)
=
-2\Big(\sum\limits_{i=1}^{N}\varphi_{i1},
\sum\limits_{i=1}^{N}\varphi_{i2},\ldots,
\sum\limits_{i=1}^{N}\varphi_{iM}\Big);
\\[0.2 in]
(\nabla g)(x)
=
2(x_{1},x_{2},\ldots,x_{M}).
\end{array}
\end{equation*}
We solve the system 
of equations $(\nabla f)(x)=\lambda[(\nabla g)(x)]$, $g(x)=1$ 
where $\lambda\in\mathbb{R}$ to obtain
\begin{equation}\label{ENTFAA2.34}
x_{1}=-\frac{1}{\lambda}\sum\limits_{i=1}^{N}\varphi_{i1},
\quad
x_{2}=-\frac{1}{\lambda}\sum\limits_{i=1}^{N}\varphi_{i2},
\quad\ldots,\quad
x_{M}=-\frac{1}{\lambda}\sum\limits_{i=1}^{N}\varphi_{iM}.
\end{equation}
Substituting these coordinates into the function $g(x)=1$ gives 
\begin{equation*}
1
=
\frac{1}{\lambda^{2}}\sum\limits_{j=1}^{M}
\Big(\sum\limits_{i=1}^{N}\varphi_{ij}\Big)^{2}.
\end{equation*}
By Lemma~\ref{ENTFAA2.21}, we may write the above equation 
as 
\begin{equation*}
\lambda^{2}
=
\sum\limits_{i=1}^{N}\|\varphi_{i}\|^{2}+
2\sum\limits_{1\leq i<k\leq N}
\langle\varphi_{i},\varphi_{k}\rangle
=
N+2\sum\limits_{1\leq i<k\leq N}
\langle\varphi_{i},\varphi_{k}\rangle,
\end{equation*}
which implies 
\begin{equation*}
\lambda
=
\pm\Big(N+2\sum\limits_{1\leq i<k\leq N}
\langle\varphi_{i},\varphi_{k}\rangle\Big)^{1/2}.
\end{equation*}
Substituting the coordinates in \eqref{ENTFAA2.34} into the 
function $f$ given in \eqref{ENTFAA2.30} and using 
Lemma~\ref{ENTFAA2.21} yields
\begin{align*}
f(x)
&=
2N-2
\sum\limits_{i=1}^{N}\sum\limits_{j=1}^{M}
\Big(-\frac{1}{\lambda}\sum\limits_{k=1}^{N}\varphi_{kj}\Big)
\varphi_{ij}
\nonumber\\[2 pt]
&=
2N+\frac{2}{\lambda}
\Big[\sum\limits_{i=1}^{N}\sum\limits_{j=1}^{M}\varphi_{ij}
\Big(\sum\limits_{k=1}^{N}\varphi_{kj}\Big)\Big]
\nonumber\\[2 pt]
&=
2N+\frac{2}{\lambda}
\Big[\sum\limits_{i=1}^{N}\|\varphi_{i}\|^{2}+
2\sum\limits_{1\leq i<k\leq N}
\langle\varphi_{i},\varphi_{k}\rangle\Big]
\nonumber\\[2 pt]
&=
2N+\frac{2}{\lambda}[\lambda^{2}]
=
2(N+\lambda).
\end{align*}
Thus, the global extrema for the function $f$ are $2(N+\lambda)$. 
Since $\lambda$ depends on 
$\langle\varphi_{i},\varphi_{k}\rangle$ $($and \emph{not} on 
$|\langle\varphi_{i},\varphi_{k}\rangle|)$ for each $i\neq k$, then $\lambda$ reaches its extrema when $\langle\varphi_{i},\varphi_{k}\rangle=c$ where $c>0$. Since 
$\{\varphi_{i}\}_{i=1}^{N}$ is an equiangular tight 
frame,
$\displaystyle c=\sqrt{\frac{N-M}{M(N-1)}}$. Under these assumptions,
\begin{equation*}
\lambda
=
\pm\sqrt{N+2\frac{N(N-1)}{2}c}
=
\pm\sqrt{N[1+(N-1)c]}.
\end{equation*}
With $\lambda$ as such, the result follows.
\end{proof}

\begin{remark}\label{ENTFAA2.40}
In Proposition~\ref{ENTFAA2.26}, the lower bound is always
\emph{positive}. Indeed, $2\big(N-\sqrt{N[1+(N-1)c]}\big)>0$ if and only if $c<1$ which is satisfied a priori.
\end{remark}

We establish slightly weaker bounds for 
Proposition~\ref{ENTFAA2.26} in the following corollary. Although these bounds are weaker, they are more accessible.

\begin{corollary}\label{ENTFAA2.42}
Let $\{\varphi_{i}\}_{i=1}^{N}$ be an equiangular tight 
frame in $\mathbb{R}^{M}$. Then for any unit norm 
$x\in\mathbb{R}^{M}$, we have 
\begin{align}\label{ENTFAA2.43}
2N(1-\sqrt{2c})
&<
2\big(N-\sqrt{N[1+(N-1)c]}\big)
\nonumber\\[2 pt]
&\leq
\sum\limits_{i=1}^{N}\lVert x-\varphi_{i}\rVert^{2}
\nonumber\\[2 pt]
&\leq
2\big(N+\sqrt{N[1+(N-1)c]}\big)
\nonumber\\[2 pt]
&<
2N(1+\sqrt{2c}),
\end{align}
where $c:=\sqrt{\frac{N-M}{M(N-1)}}$; i.e., 
$c=\lvert\langle\varphi_{i},\varphi_{j}\rangle\rvert,
\;\forall\,i,j\in\{1,\ldots,N\}$ such that $i\neq j$.
\end{corollary}

\begin{proof}
First note 
\begin{equation*}
c
=
\sqrt{\frac{N-M}{M(N-1)}}
\geq
\sqrt{\frac{1}{M(N-1)}}
\geq
\sqrt{\frac{1}{N(N-1)}}
>
\frac{1}{N}.
\end{equation*}
For the right hand side of the inequality in \eqref{ENTFAA2.43}, 
we have 
\begin{align}\label{ENTFAA2.45}
2\big(N+\sqrt{N[1+(N-1)c]}\big)
&=
2\Big(N+N\sqrt{\frac{1}{N}+\frac{N-1}{N}c}\Big)
\nonumber\\[2 pt]
&<
2\Big(N+N\sqrt{\frac{1}{N}+c}\Big)
\nonumber\\[2 pt]
&<
2(N+N\sqrt{2c})
\nonumber\\[2 pt] 
&=
2N(1+\sqrt{2c}).
\end{align}
For the left hand side of the inequality in 
\eqref{ENTFAA2.43}, we have 
\begin{align}\label{ENTFAA2.46}
2\big(N-\sqrt{N[1+(N-1)c]}\big)
&=
2\Big(N-N\sqrt{\frac{1}{N}+\frac{N-1}{N}c}\Big)
\nonumber\\[2 pt]
&>
2\Big(N-N\sqrt{\frac{1}{N}+c}\Big)
\nonumber\\[2 pt]
&>
2(N-N\sqrt{2c})
\nonumber\\[2 pt]
&=
2N(1-\sqrt{2c}),
\end{align}
as wanted. The last order of business is to check that 
$2N(1-\sqrt{2c})>0$ which happens if and only if $c<1/2$. 
We consider the following scenarios:
\vskip12pt
\noindent\underline{Case 1:} 
$M \geq 4$. We proceed by way of contradiction. Note that $c\geq 1/2$ if and only if 
$\displaystyle {N\geq\frac{NM}{4}+\frac{3M}{4}}$. If $M\geq 4$ and 
$c\geq 1/2$ then \[{N\geq\frac{NM}{4}+\frac{3M}{4}\geq N+3},\] 
a contradiction. Therefore $c<1/2$ when $M\geq 4$ 
$($which forces $N\geq 4)$. 
\vskip12pt
\noindent\underline{Case 2:} 
$M=3.$ Note that $c<1/2$ if and only if $4N-3M-MN<0$. Thus, for 
$M=3$, we have $4N-3(3)-3N<0$ if and only if $N<9$. Recall in 
$\mathbb{R}^{3}$, it is known that the maximum number of vectors 
in an equiangular unit norm tight frame is $6$
 (see www.framerc.org/equiangular Frames). Hence, 
$c<1/2$ when $M=3$ $($and, hence, $3\leq N\leq 6<9)$.
\vskip12pt
\noindent\underline{Case 3:} 
$M=2.$ Recall that in $\mathbb{R}^{2}$, it is known that the maximum 
number of vectors in an equiangular unit norm tight frame is $3$ (see www.framerc.org/equiangular Frames). 
Also, from \underline{Case 2:}, we know $c<1/2$ if and 
only if $4N-3M-MN<0$. For $M=2$, this becomes $4N-3(2)-2N<0$ 
which happens if and only if $N<3$. Consequently, $c<1/2$ when 
$M=2$ and $N=2$. 

This finishes the proof of Corollary~\ref{ENTFAA2.42}.
\end{proof}

\begin{remark}

For the special case $M=2$ and $N=3$, it follows 
that $c=1/2$ $($see \eqref{ENTFAA2.20}$)$ and 
\begin{equation*}
\frac{1}{N}+c
=
\frac{1}{3}+\frac{1}{2}
=
\frac{5}{6}
=
\frac{5}{3}\cdot\frac{1}{2}
=
\frac{5}{3}c.
\end{equation*}
Thus, we actually get improved bounds for the case when $M=2$ and 
$N=3$; specifically, $($see \eqref{ENTFAA2.45} and 
\eqref{ENTFAA2.46}$)$
\begin{align*}
2N\Big(1-\sqrt{\frac{5}{3}c}\Big)
&\leq
2\big(N-\sqrt{N[1+(N-1)c]}\big)
\nonumber\\[2 pt]
&\leq
\sum\limits_{i=1}^{N}\|x-\varphi_{i}\|^{2}
\nonumber\\[2 pt]
&\leq
2\big(N+\sqrt{N[1+(N-1)c]}\big)
\nonumber\\[2 pt]
&\leq
2N\left(1+\sqrt{\frac{5}{3}c}\right).
\end{align*}
\end{remark}

The next corollary gives a surprising identity exhibited by 
\emph{simplex frames}. Recall that up to multiplication by a unitary operator and switching (replacing a vector by its additive inverse) there is only one unit norm tight frame with $M+1$ elements in $\RR^M$ \cite{CK2}. This frame can be obtained in the following way. Let $\{e_i\}_{i=1}^{M+1}$ be the standard orthonormal basis for $\RR^{M+1}$ and let $P$ be the rank one orthogonal projection onto $\mathrm{span}\left(\sum_{i=1}^{M+1}e_i\right)$. Then the vectors 
\[\{\phi_i\}_{i=1}^{M+1}:=\left\{\frac{(I-P)e_i}{\|(I-P)e_i\|}\right\}_{i=1}^{M+1}\]
form an equiangular tight frame for $\RR^M$. This frame is commonly referred to as the simplex frame in $\RR^M$.
\begin{corollary}
Let $\{\varphi_{i}\}_{i=1}^{M+1}$ be the simplex frame in 
$\mathbb{R}^M$. Then for any unit norm $x\in\mathbb{R}^{M}$, we 
have 
\begin{equation*}
\sum\limits_{i=1}^{M+1}\|x-\varphi_{i}\|^{2}
=
2(M+1).
\end{equation*}
\end{corollary}

\begin{proof}
This follows immediately from Proposition~\ref{ENTFAA2.26} and 
the fact that $\displaystyle {c=-\frac{1}{M}=-\frac{1}{N-1}}$  for the simplex in $\mathbb{R}^{M}$.
\end{proof}

There is a significant generalization of the above corollary
for the case where the frame vectors sum to zero.

\begin{theorem} 
Let $\{\varphi_i\}_{i=1}^N$ be a unit norm tight frame in $\mathcal{H}_M$ and assume $\sum_{i=1}^N \varphi_i=0$. Then $\sum_{i=1}^N \|x-\varphi_i\|^2 = 2N$ for any unit norm $x \in \mathcal{H}_M$.
\end{theorem}

\begin{proof}
For any unit norm $x\in \mathcal{H}_M$, we have
\begin{eqnarray*}
\sum_{i=1}^N\|x-\varphi_i\|^2 &=& \sum_{i=1}^N \|x\|^2 + \|\varphi_i\|^2 - \langle x,\varphi_i \rangle - \langle \varphi_i, x\rangle \\
&=& \sum_{i=1}^N( 2 - 2 {\rm Re}(\langle x, \varphi_i \rangle )) \\&=& 2N - 2{\rm Re} \langle x, \sum_{i=1}^N \varphi_i
\rangle \\
&=& 2N - 2{\rm Re}\langle x, 0\rangle = 2N.
\end{eqnarray*} 
\end{proof}

Finally, in this setting, we can get very accurate approximations for the sums of squared products of distances 
from vectors to the frame vectors.

\begin{theorem}
Let $\{\varphi_{i}\}_{i=1}^{N}$ be a unit norm tight frame in 
$\mathbb{R}^{M}$ and assume 
$\sum\limits_{i=1}^{N}\varphi_{i}=0$. Then for any unit norm $x\in \mathcal{H}_M$, 
\begin{equation*}
4N\Big(1-\frac{1}{M}\Big)
\leq
\sum\limits_{i=1}^{N}
\|x-\varphi_{i}\|^{2}\|y-\varphi_{i}\|^{2}
\leq
4N\Big(1+\frac{1}{M}\Big).
\end{equation*}
\end{theorem}

\begin{proof}
First, note that for each $i\in\{1,\ldots,N\}$ and each 
$x\in\mathbb{R}^{M}$ unit norm, we have 
\begin{equation*}
\|x-\varphi_{i}\|^{2}
=
\|x\|^{2}+\|\varphi_{i}\|^{2}-2\langle x,\varphi_{i}\rangle
=
2\big(1-\langle x,\varphi_{i}\rangle\big).
\end{equation*}
Using this, we may write 
\begin{align*}
\sum\limits_{i=1}^{N}
\|x-\varphi_{i}\|^{2}\|y-\varphi_{i}\|^{2}
&=
\sum\limits_{i=1}^{N}
2\big(1-\langle x,\varphi_{i}\rangle\big)
2\big(1-\langle y,\varphi_{i}\rangle\big)
\nonumber\\[2 pt]
&=
4\sum\limits_{i=1}^{N}
\big(1-\langle x,\varphi_{i}\rangle-\langle y,\varphi_{i}\rangle
+\langle x,\varphi_{i}\rangle\langle y,\varphi_{i}\rangle\big)
\nonumber\\[2 pt]
&=
4\Big(N-0-0
+\big\langle x,\sum\limits_{i=1}^{N}
\langle y,\varphi_{i}\rangle\varphi_{i}\big\rangle\Big)
\nonumber\\[2 pt]
&=
4\Big(N+\big\langle x,\frac{N}{M}y\big\rangle\Big)
=
4N\Big(1+\frac{1}{M}\langle x,y\rangle\Big).
\end{align*}
Using Cauchy-Schwarz and the fact that $x,y\in\HH_{M}$ are unit 
norm, we may deduce $-1\leq\langle x,y\rangle\leq1$. Combining 
this with the above calculation gives the desired result.
\end{proof}





\end{document}